\documentclass[a4paper,reqno,11pt,oneside]{amsart}
\usepackage{amsmath,geometry}
\usepackage{color}
\usepackage{latexsym}
\usepackage{amsmath}
\usepackage{amsfonts}
\usepackage{amssymb}
\usepackage{graphicx}
\usepackage{esint}
\usepackage{slashed}

\numberwithin{equation}{section}

\newtheorem{theorem}{Theorem}[section]

\newtheorem{proposition}[theorem]{Proposition}

\newtheorem{corollary}[theorem]{Corollary}

\newcommand{\R}{\mathbb{R}}
\newcommand{\N}{\mathbb{N}}

\title{The quasi-linear Brezis-Nirenberg problem in low dimensions}

\author{Sabina Angeloni}
\address{Sabina Angeloni, Dipartimento di Matematica e Fisica, Universit\`a degli Studi Roma Tre, Largo S.~Leonardo Murialdo 1, Roma 00146, Italy.}
\email{sabina.angeloni@uniroma3.it}

\author{Pierpaolo Esposito}
\address{Pierpaolo Esposito, Dipartimento di Matematica e Fisica, Universit\`a degli Studi Roma Tre, Largo S.~Leonardo Murialdo 1, Roma 00146, Italy.}
\email{esposito@mat.uniroma3.it}

\begin{document}

\begin{abstract}
We discuss existence results for a quasi-linear elliptic equation of critical Sobolev growth \cite{brezisnirenberg,gueddaveron2} in the low-dimensional case, where the problem has a global character which is encoded in sign properties of the ``regular" part for the corresponding Green's function as in \cite{Dru,Esp}.
\end{abstract}

\maketitle

\section{Introduction}
Let $\Omega$ be  a bounded domain in $\R^N$, $N \geq 2$. Given $1< p<N$ and $\lambda<\lambda_1$,  let us discuss existence issues for the quasilinear problem
\begin{equation} \label{BNproblem_phdthesis}
\begin{cases}
-\Delta_p u = \lambda u^{p-1} + u^{p^*-1} \qquad &\text{in } \Omega \\
u>0 &\text{in } \Omega \\
u=0 &\text{on } \partial \Omega,
\end{cases}
\end{equation}
where $\Delta_p{(\cdot)}=\text{div}\bigl(|\nabla(\cdot)|^{p-2} \nabla(\cdot)\bigr)$ is the $p$-Laplace operator, $p^*=\frac{Np}{N-p}$ is the so-called critical Sobolev exponent and $\lambda_1$ is the first eigenvalue of $-\Delta_p$ given by
$$\lambda_1= \inf_{u \in W_0^{1,p}(\Omega) \setminus \{0\}}  \frac{\int_\Omega |\nabla u|^p}{\int_\Omega |u|^p}.$$
Since $W_0^{1,p}(\Omega) \subset L^{p^*}(\Omega)$ is a continuous but non-compact embedding, standard variational methods fail to provide solutions of \eqref{BNproblem_phdthesis} by minimization of the Rayleigh quotient
\begin{equation*}
Q_\lambda(u)= \frac{ \int_\Omega |\nabla u|^p - \lambda\int_\Omega |u|^p}{(\int_\Omega |u|^{p^*})^\frac{p}{p^*}}, \quad u \in W_0^{1,p}(\Omega) \setminus \{0\}.
\end{equation*}
Setting
$$ S_\lambda= \inf \left\{ Q_\lambda(u) \colon u \in W_0^{1,p}(\Omega)\setminus \{ 0 \} \right\},$$
it is known that $S_0$ coincides with the best Sobolev constant for the embedding $\mathcal D^{1,p}(\mathbb{R}^N) \subset L^{p^*}(\mathbb{R}^N)$ and then is never attained since independent of $\Omega$. Moreover, by a Pohozaev identity \eqref{BNproblem_phdthesis}$_{\lambda=0}$ is not solvable on star-shaped domains, see \cite{brezisnirenberg,gueddaveron2}. The presence of the perturbation term $\lambda u^{p-1}$ in \eqref{BNproblem_phdthesis} can possibly restore compactness and produce minimizers for $Q_\lambda$, as shown for all $\lambda>0$ first by Brezis and Nirenberg \cite{brezisnirenberg} in the semi-linear case when $N\geq 4$ and then by Guedda and Veron \cite{gueddaveron2} when $N\geq p^2$. 

\medskip \noindent Let us discuss now the low-dimensional case $p<N < p^2$.  In the semi-linear situation $p=2$ it corresponds to $N=3$ and displays the following special features: according to \cite{brezisnirenberg},  problem \eqref{BNproblem_phdthesis} is solvable on a ball precisely for $\lambda \in (\frac{\lambda_1}{4},\lambda_1)$ and then, for the minimization problem on a general domain $\Omega$, there holds
$$\lambda_*= \inf \left\{ \lambda \in (0, \lambda_1) \colon S_\lambda <S_0 \right\} \geq \frac{1}{4} \lambda_1(B)
=\frac{\pi^2}{4}\Big(\frac{3|\Omega|}{4\pi}\Big)^{-\frac23}$$
through a re-arrangement argument, where $B$ is the ball having the same measure of $\Omega$. Since $S_\lambda$ decreases in a continuous way from $S_0$ to $0$ as $\lambda$ ranges in $[0,\lambda_1)$, notice that $S_\lambda=S_0$ for $\lambda \in [0,\lambda_*]$, $S_\lambda<S_0$ for $\lambda \in (\lambda_*,\lambda_1)$ and $S_\lambda$ is not attained for $\lambda \in [0,\lambda_*)$. A natural question concerns the case $\lambda=\lambda_*$ and the following  general answer 
\begin{equation} \label{hypothesisSlambdastar}
S_{\lambda_*} \text{ is not achieved}
\end{equation}
has been given by Druet \cite{Dru}, with an elegant proof which unfortunately seems not to work for $p>2$. A complete characterization for the critical parameter $\lambda_*$ then follows through a blow-up approach crucially based on \eqref{hypothesisSlambdastar}. 

\medskip \noindent We use here some of the results in \cite{EsAn1} - precisely reported in Section $2$ for reader's convenience - as a crucial ingredient to treat the quasilinear Brezis-Nirenberg problem \eqref{BNproblem_phdthesis} in the low-dimensional case $p<N<p^2$.  Given $x_0 \in \Omega$ and $\lambda<\lambda_1$,  introduce the Green function $G_\lambda(\cdot,x_0)$ as a positive solution to
\begin{equation} \label{Gproblem_phdthesis}
\begin{cases}
-\Delta_p G  - \lambda G^{p-1}= \delta_{x_0} \qquad &\text{in } \Omega\\
G=0 &\text{on } \partial \Omega.
\end{cases}
\end{equation}
Since uniqueness of $G_\lambda(\cdot,x_0)$ is just known for $p\geq 2$,  hereafter we will just consider the case $p\geq 2$.  If $\omega_N$ denotes the measure of the unit ball in $\R^N$, recall that the fundamental solution
\begin{equation} \label{gamma_phdthesis}
\Gamma(x,x_0)= C_0|x-x_0|^{-\frac{N-p}{p-1}}, \quad C_0=\frac{p-1}{N-p}(N\omega_N)^{-\frac{1}{p-1}},
\end{equation}
solves $ -\Delta_p \Gamma=\delta_{x_0}$ in $\R^N$.  The function $H_\lambda(x,x_0)=G_\lambda(x,x_0)-\Gamma(x,x_0)$ is usually referred to as the ``regular" part of $G_\lambda(\cdot,x_0)$ but is just expected to be less singular than $\Gamma(x,x_0)$ at $x_0$.  

\medskip \noindent The complete characterization in \cite{Dru} for $\lambda_*$  still holds in the quasi-linear case,  as stated by the following main result.
\begin{theorem} \label{existenceBNproblem}
Let $2\leq p<N<2p$ and $0<\lambda<\lambda_1$. The implications (i) $\Rightarrow$ (ii) $\Rightarrow$ (iii) do hold, where
\begin{itemize}
\item [(i)] 
there exists $x_0 \in \Omega$ such that $H_\lambda(x_0,x_0)>0$ 
\item[(ii)] $S_\lambda<S_0$ 
\item[(iii)] $S_\lambda$ is attained.
\end{itemize}
Moreover, the implication (iii) $\Rightarrow$ (i) does hold under the assumption \eqref{hypothesisSlambdastar} and in particular $\lambda_*>0$.
\end{theorem}
Some comments are in order.  Assumption $N<2p$ is crucial here to guarantee that $H_\lambda(\cdot,x_0)$ is H\"older continuous at $x_0$, see \cite{EsAn1}.  When $2p\leq N<p^2$ we conjecture $H_\lambda(x,x_0)$ to be mildly but still singular at $x_0$, with a behavior like $\frac{m_\lambda(x_0)}{|x-x_0|^\alpha}$ for an appropriate $0<\alpha <\frac{N-p}{p-1}$,  and $m_\lambda(x_0)$ to play the same role as $H_\lambda(x_0,x_0)$ in Theorem \ref{existenceBNproblem}.  The quantity $m_\lambda(x_0)$ is usually referred to as the mass associated to $G_\lambda(\cdot,x_0)$ and appears in several contexts, see for example \cite{GhRo1,GhRo2,Sch,ScYa1,ScYa2}.  Notice that in the semilinear case $p=2$ the range $2p\leq N<p^2$ is empty and such a situation doesn't show up in \cite{Dru}.

\medskip \noindent The implication $(iii)$ $\Rightarrow$ $(i)$ follows by a blow-up argument once  \eqref{hypothesisSlambdastar} is assumed. To this aim, we first extend the pointwise blow-up theory in \cite{DHR} to the quasi-linear context, a fundamental tool in the description of blow-up  phenomena whose relevance goes beyond Theorem \ref{existenceBNproblem} and which completely settles some previous partial results \cite{AuLi,DjDr,Dru1} in this direction. Once sharp pointwise blow-up estimates are established, a major difficulty appears in the classical use of Pohozaev identities: written on small balls around the blow-up point as the radius tends to zero, they rule both the blow-up speed and the blow-up location since boundary terms in such identities can be controlled thanks to the property $\nabla H_\lambda (\cdot,x_0) \in L^\infty(\Omega)$. Clearly valid in the semi-linear situation, such gradient $L^\infty-$bound is completely missing in the quasi-linear context but surprisingly the correct answer can still be found by a different approach, based on a suitable approximation scheme for $G_\lambda(\cdot,x_0)$. At the same time, we provide a different proof of some facts in \cite{Dru} in order to avoid some rough arguments concerning  the limiting problems on halfspaces, when dealing with boundary blow-up.
 
\medskip \noindent Under the assumption \eqref{hypothesisSlambdastar}, in the proof of Theorem \ref{existenceBNproblem} we will show that $H_{\lambda_*}(x_0,x_0)=0$ for some $x_0 \in \Omega$, a stronger property than the validity of the implication $(iii)$ $\Rightarrow$ $(i)$ since $H_\lambda(x,x)$ is strictly increasing in $\lambda$ for all $x \in \Omega$.  Since $S_0$ is not attained, notice that \eqref{hypothesisSlambdastar} always holds if $\lambda_*=0$ and then $\lambda_*>0$ follows by the property $H_0(x_0,x_0)<0$ for all $x_0 \in \Omega$. Moreover, since 
\begin{equation} \label{63859} \sup_{x \in \Omega} H_{\lambda_*}(x,x)=\max_{x \in \Omega} H_{\lambda_*}(x,x)=0,
\end{equation}
by monotonicity of $H_\lambda$ in $\lambda$ and under the assumption \eqref{hypothesisSlambdastar} the critical parameter $\lambda_*$ is the first unique value of $\lambda>0$ attaining \eqref{63859} and can be re-written as
$$\lambda_*= \sup \left\{ \lambda \in (0, \lambda_1) \colon H_\lambda(x,x) <0 \hbox{ for all }x \in \Omega\right\}.$$

\medskip \noindent In Section $2$ we recall some facts from \cite{EsAn1} that will be used throughout the paper and prove some useful convergence properties. The implication $(i)$ $\Rightarrow$ $(ii)$ is established in Section $3$ by the expansion of $Q_\lambda(PU_{\epsilon,x_0})$ along the ``bubble" $PU_{\epsilon,x_0}$ concentrating at $x_0$ as $\epsilon \to 0$
and integral identities of Pohozaev type for $G_\lambda(\cdot,x_0)$, crucial for a fine asymptotic analysis, are also derived. Section $4$ is devoted to develop the blow-up argument along with sharp pointwise estimates to establish the final part in Theorem \ref{existenceBNproblem}.

\section{Some preliminary facts}
For reader's convenience,  let us collect here some of the results in \cite{EsAn1}.  To give the statement of Theorem \ref{existenceBNproblem} a full meaning, we need  a general theory for problem \eqref{Gproblem_phdthesis}, as stated in the following result.
\begin{theorem} \label{mainth} \cite{EsAn1}
Let $1<p\leq N$ and $\lambda<\lambda_1$. Assume $p\geq 2$ and $N<2p$ if $\lambda \not=0$. Then problem \eqref{Gproblem_phdthesis} has a positive solution $G_\lambda(\cdot,x_0)$ so that 
\begin{equation} \label{1656}
\nabla H_\lambda (\cdot,x_0) \in L^{\bar q}(\Omega),\quad \bar q =\frac{N(p-1)}{N-1},
\end{equation} 
which is unique when either $\lambda=0$ or $\lambda \not=0$ and \eqref{1656} holds.  Moreover
\begin{itemize}
\item if $p\geq 2$, given $M>0$, $q_0>\frac{N}{p}$ and $p_0 \geq 1$ there exists $C>0$ so that
\begin{equation} \label{corollarylocalboundH3}
\| H+c \|_{\infty,B_r(x_0)} \leq C (r^{-\frac{N}{p_0}} \| H+c \|_{p_0,B_{2r}(x_0)} +r^\frac{pq_0-N}{q_0(p-1)} \|f \|^\frac{1}{p-1}_{q_0,B_{2r}(x_0)}) 
\end{equation}
for all $\epsilon^{p-1}\leq r\leq \frac{1}{4}\hbox{dist}(x_0,\partial \Omega)$ and all solution $G=\mathit{\Gamma}+H$ to 
\begin{equation} \label{18300}
-\Delta_p  G+\Delta_p \mathit{\Gamma}= f \qquad \text{in } \Omega \setminus \{x_0\}
\end{equation}
so that $\nabla H \in L^{\bar q}(\Omega)$, $\frac{|x-x_0|^\frac{1}{p-1}}{M(\epsilon^p+|x-x_0|^\frac{p}{p-1})^\frac{N}{p}}\leq  |\nabla \mathit{\Gamma}| \leq M|\nabla \Gamma|(x,x_0)$ and $|c|+\|H \|_\infty+ \| f \|^\frac{1}{p-1}_{q_0} \leq M$, where $\Gamma(\cdot,x_0)$ is given by \eqref{gamma_phdthesis};
\item  $\lambda G_\lambda^{p-1} \in L^{q_0}(\Omega)$ for $q_0>\frac{N}{p}$ and $H_\lambda(\cdot,x_0)$ is a continuous function in $\overline \Omega$ satisfying 
\begin{equation} \label{holdercontinuity}
|H_\lambda(x,x_0)- H_\lambda(x_0,x_0)| \leq C |x-x_0|^\alpha \quad \forall \ x \in \Omega 
\end{equation}
for some $C>0$, $ \alpha \in (0,1)$ with $H_\lambda(x_0,x_0)$ strictly increasing in $\lambda$.
\end{itemize}
\end{theorem}
Notice that the first part in Theorem \ref{mainth} has been established in \cite{kichenassamyveron}. Let us stress that the condition $f \in L^{q_0}(\Omega)$ for some $q_0>\frac{N}{p}$, which is valid for $f=\lambda G_\lambda^{p-1}$ when $N<2p$ if $\lambda \not=0$,  is  a natural condition on the R.H.S. of the difference equation \eqref{18300} to prove $L^\infty-$bounds on $H$ as it arises for instance in the Moser iterative argument adopted in \cite{serrin}.  In this respect,  observe that also in the semilinear case $H_\lambda(\cdot, x_0)$ is no longer regular at $x_0$ when $4=2 p \leq N$.  

\medskip \noindent The following a-priori estimates are the basis of Theorem \ref{mainth} and will be crucially used here to establish some accurate pointwise blow-up estimates.
\begin{proposition} \label{lemmaGj} \cite{EsAn1} Let $2\leq p\leq N$. Assume that $a_n \in L^\infty(\Omega)$, $f_n \in L^1(\Omega)$ and $g_n,\hat g_n $ satisfy
$$g_n,\hat g_n \in L^\infty(\Omega) \cap W^{1,p}(\Omega) \hbox{ $p-$harmonic in }\Omega, \ g_n,\hat g_n \hbox{ non-constant unless }0$$
and
$$\lim_{n \to +\infty} \|a_n-a\|_\infty=0 \hbox{ with }\sup_\Omega a<\lambda_1, \quad \sup_{n \in \mathbb N} \left[\|f_n\|_1 +\|g_n\|_\infty + \|\hat g_n\|_\infty \right]<+\infty.$$
If $u_n \in W^{1,p}_{g_n}(\Omega)$ solves $-\Delta_p u_n - a_n|u_n|^{p-2}u_n= f_n$ in $\Omega$, then $\displaystyle \sup_{n \in \mathbb N} \|u_n\|_{p-1}<+\infty$ and, if $g_n = g$, the sequence $u_n$ is pre-compact in $W^{1,q}(\Omega)$ for all $1\leq q < \bar q$. Moreover, if $N<2p$, $a_n =\lambda_n \in \mathbb{R}$ and $\hat u_n \in W^{1,p}_{\hat g_n}(\Omega)$ solves $-\Delta_p \hat u_n = f_n$ in $\Omega$, then $\displaystyle \sup_{n \in \mathbb{N}} \|u_n - \hat u_n\|_\infty<\infty$.
\end{proposition} 
We will also make use of the following general form of comparison principle.
\begin{proposition} \label{wcp} \cite{EsAn1}
Let $2\leq p \leq N$ and $a ,f_1,f_2 \in L^\infty (\Omega)$. Let $u_i \in C^1(\bar \Omega)$, $i=1,2$, be solutions to 
$$-\Delta_p u_i-a u_i^{p-1}=f_i \quad \text{ in } \Omega$$
so that 
$$u_i>0 \hbox{ in }\Omega, \quad \frac{u_1}{u_2} \leq C  \hbox{ near } \partial \Omega  $$
for some $C>0$. If $f_1 \leq f_2$ with $f_2 \geq 0$ in $\Omega$ and $u_1 \leq u_2$ on $\partial \Omega$, then $u_1\leq u_2$ in $\Omega$.
\end{proposition}
Let us introduce now a special approximation scheme for $G_\lambda(\cdot,x_0)$, which is particularly suited for the problem we are interested in.  The so-called standard bubbles 
\begin{equation} \label{definitionbubble_phdthesis}
U_{\epsilon,x_0} (x)= C_1 \biggl( \frac{\epsilon}{\epsilon^p + |x-x_0|^\frac{p}{p-1}}\biggr)^\frac{N-p}{p} \quad \epsilon>0,\  x_0 \in \R^N, \ C_1=N^\frac{N-p}{p^2} \bigl( \frac{N-p}{p-1}\bigr)^\frac{(p-1)(N-p)}{p^2},
\end{equation}
are the extremals of the Sobolev inequality
$$S_0 \left( \int_{\mathbb{R}^N} |u|^{p^*}\right)^{\frac{p}{p^*}} \leq  \int_{\mathbb{R}^N}|\nabla u|^{p}, \quad u \in \mathcal D^{1,p}(\mathbb{R}^N),$$
and the unique entire solutions in $\mathcal D^{1,p}(\mathbb{R}^N)$ of
\begin{equation} \label{equationbubble_phdthesis}
-\Delta_p U=U^{p^* -1} \quad \text{in } \R^N,
\end{equation}
see \cite{sciunzi,Sci,vetois}.  For $\lambda<\lambda_1$ consider its projection $PU_{\epsilon,x_0}$ in $\Omega$, as the solution of 
\begin{equation} \label{PUepsilonproblem_phdthesis}
\begin{cases}
-\Delta_p PU_{\epsilon,x_0} = \lambda PU_{\epsilon,x_0}^{p-1} + U_{\epsilon,x_0}^{p^*-1} \quad &\text{in } \Omega \\
PU_{\epsilon,x_0} >0 &\text{in } \Omega \\
PU_{\epsilon,x_0} = 0 &\text{on } \partial \Omega.
\end{cases}
\end{equation}
Letting $G_{\epsilon,x_0}= \frac{C_0}{C_1} \epsilon^{-\frac{N-p}{p}} PU_{\epsilon,x_0}$ with  $C_0$ given by \eqref{gamma_phdthesis},  decompose it as $G_{\epsilon,x_0} = \Gamma_{\epsilon,x_0} + H_{\epsilon,x_0}$, where
\begin{equation} \label{11011}
\Gamma_{\epsilon,x_0} = \frac{C_0}{C_1} \epsilon^{-\frac{N-p}{p}} U_{\epsilon,x_0}=\frac{C_0}{(\epsilon^p + |x-x_0|^\frac{p}{p-1})^\frac{N-p}{p}} \to \Gamma(x,x_0)
\end{equation}
in $C^1_{\hbox{loc}}(\bar \Omega \setminus \{x_0\})$ as $\epsilon \to 0$.  Since 
\begin{equation} \label{01719}
f_{\epsilon,x_0}:= -\Delta_p \Gamma_{\epsilon,x_0}= \bigl(\frac{C_0}{C_1} \epsilon^{-\frac{N-p}{p}}\bigr)^{p-1} U_{\epsilon,x_0}^{p^*-1} =
\frac{C_0^{p-1} C_1^{\frac{p^2}{N-p}} \epsilon^p}{(\epsilon^p + |x-x_0|^\frac{p}{p-1})^{N-\frac{N-p}{p}}} \to 0 \end{equation} 
in $C_{\hbox{loc}}(\bar \Omega \setminus \{x_0\})$ and
$$\int_\Omega f_{\epsilon,x_0}= - \int_{\partial \Omega} |\nabla \Gamma_{\epsilon,x_0}|^{p-2}\partial_\nu \Gamma_{\epsilon,x_0}
\to - \int_{\partial \Omega} |\nabla \Gamma|^{p-2}(x,x_0)\partial_\nu \Gamma(x,x_0) d \sigma(x)=1$$
as $\epsilon \to 0$ in view of \eqref{equationbubble_phdthesis} and \eqref{11011}, notice that $f_{\epsilon,x_0}  \rightharpoonup \delta_{x_0}$ weakly in the sense of measures in $\Omega$ as $\epsilon \to 0$ and $G_{\epsilon,x_0}$ solves 
\begin{equation} \label{Gepsilonproblem_phdthesis}
\begin{cases}
-\Delta_p G_{\epsilon,x_0}= \lambda G_{\epsilon,x_0}^{p-1} + f_{\epsilon,x_0} \quad &\text{in } \Omega \\
G_{\epsilon,x_0} >0 &\text{in } \Omega \\
G_{\epsilon,x_0}=0 &\text{on } \partial \Omega.
\end{cases}
\end{equation}
Thanks to Theorem \ref{mainth} and Proposition \ref{lemmaGj}  we can now establish the following convergence result.
\begin{proposition} \label{propositionuniformconvergenceHepsilon}
Let $2\leq p\leq N$ and assume $N<2p$ if $\lambda \not=0$. Then there holds 
\begin{equation} \label{uniformconvergenceHepsilon}
H_{\epsilon,x_0} \to H_\lambda(\cdot,x_0) \quad \hbox{in }C(\bar \Omega)
\end{equation}
as $\epsilon \to 0$.
\end{proposition}
\begin{proof}
By Proposition \ref{lemmaGj} we can find a subsequence $\epsilon_n \to 0$ so that $G_{\epsilon_n,x_0} \to G$ in $W_0^{1,q}(\Omega)$ as $n \to +\infty$ for all $1\leq q< \bar{q}$, where $G=\Gamma(x,x_0)+H$ is a solution of \eqref{Gproblem_phdthesis} for some $H$ in view of \eqref{11011} and \eqref{Gepsilonproblem_phdthesis}.  In particular, if $\lambda \not=0$ by the Sobolev embedding theorem there holds
\begin{equation} \label{11255}
G_{\epsilon_n,x_0} \to G \quad \hbox{ in }L^{p}(\Omega) \hbox{ as }n \to +\infty
\end{equation}
thanks to $\bar q^*>p$ in view of $N<2p\leq p^2$.  Moreover, let us rewrite \eqref{Gepsilonproblem_phdthesis} in the equivalent form:
\begin{equation} \label{Hepsilonproblem_phdthesis}
\begin{cases}
-\Delta_p(\Gamma_{\epsilon,x_0} + H_{\epsilon,x_0}) + \Delta_p\Gamma_{\epsilon,x_0}= \lambda G_{\epsilon,x_0}^{p-1} \quad &\text{in } \Omega \\
H_{\epsilon,x_0}=- \Gamma_{\epsilon,x_0} &\text{on } \partial \Omega.
\end{cases}
\end{equation}

\medskip \noindent Let us denote the solution of \eqref{Gepsilonproblem_phdthesis}$_{\lambda=0}$ by $G_{\epsilon,x_0}^0$ and set $H_{\epsilon,x_0}^0=G^0_{\epsilon,x_0}- \Gamma_{\epsilon,x_0}$.  By the uniqueness part in Theorem \ref{mainth} with $\lambda=0$ we have that 
$$G^0_{\epsilon,x_0} \to G_0(\cdot,x_0) \hbox{ in }W_0^{1,q}(\Omega)$$
as $\epsilon \to 0$,  for all $1\leq q< \bar{q}$. Moreover, since $|H_{\epsilon,x_0}^0| \leq M$ on $\partial \Omega$, by integrating \eqref{Hepsilonproblem_phdthesis} against $(H_{\epsilon,x_0}^0\mp M)_\pm$ we deduce that 
\begin{equation} \label{1630}
|H_{\epsilon,x_0}^0|\leq M \quad \hbox{ in }\Omega
\end{equation} 
in an uniform way and then $G^0_{\epsilon,x_0}$ is locally uniformly bounded in $\bar \Omega \setminus \{x_0\}$. By elliptic estimates \cite{dib,lieberman,serrin,tolksdorf} and \eqref{Gepsilonproblem_phdthesis}$_{\lambda=0}$ we deduce that 
\begin{equation}\label{15055}
G^0_{\epsilon,x_0} \hbox{ uniformly bounded in }C^{1,\alpha}_{\hbox{loc}}(\bar \Omega \setminus \{x_0\})\end{equation}
for some $\alpha \in (0,1)$. Integrating \eqref{Hepsilonproblem_phdthesis}$_{\lambda=0}$ against $\eta^p H^0_{\epsilon,x_0}$, $0\leq \eta \in C_0^\infty(\Omega)$,  we get that
$$\int_\Omega \eta^p |\nabla H^0_{\epsilon,x_0}|^p \leq p\int_\Omega \eta^{p-1} |\nabla \eta| (|\nabla \Gamma_{\epsilon,x_0}|^{p-2}+|\nabla H^0_{\epsilon,x_0}|^{p-2})|H^0_{\epsilon,x_0}||\nabla H^0_{\epsilon,x_0}|$$
and then  \eqref{1630} and Young's inequality imply that
\begin{equation}\label{15056}
\nabla H_{\epsilon,x_0}^0  \hbox{ uniformly bounded in }L^p(\Omega)
\end{equation}
in view of \eqref{15055}.  

\medskip \noindent Let us consider now the case $\lambda \not=0$.  Since
$$-\Delta_p(\Gamma_{\epsilon,x_0} + H_{\epsilon,x_0}) + \Delta_p(\Gamma_{\epsilon,x_0}+H^0_{\epsilon,x_0})= \lambda G_{\epsilon,x_0}^{p-1} \quad \text{ in } \Omega $$
with $H_{\epsilon,x_0}- H^0_{\epsilon,x_0}=0$ on $\partial \Omega$,  an integration against $H_{\epsilon,x_0}- H^0_{\epsilon,x_0}$ gives that
$$\int_\Omega  |\nabla \left(H_{\epsilon,x_0}- H^0_{\epsilon,x_0}\right) |^p \leq |\lambda| \int_\Omega G_{\epsilon,x_0}^{p-1}|H_{\epsilon,x_0}- H^0_{\epsilon,x_0}| \leq |\lambda| \|G_{\epsilon,x_0}\|_p^{p-1}
\|H_{\epsilon,x_0}- H^0_{\epsilon,x_0}\|_p
$$
thanks to the H\"older's inequality and the coercivity properties of the $p-$Laplace operator, and then
\begin{equation}\label{15056bis}
\nabla \left(H_{\epsilon_n,x_0}^0 -  H^0_{\epsilon_n,x_0} \right) \hbox{ uniformly bounded in }L^p(\Omega)
\end{equation}
in view of \eqref{11255} and Poincar\'e inequality.  A combination of \eqref{15056} and \eqref{15056bis} lead to a uniform $L^p-$bound on $\nabla H_{\epsilon_n,x_0}^0$,  showing by Fatou's lemma that $\nabla H \in L^p(\Omega)$. By Theorem \ref{mainth} we have that $G=G_\lambda(\cdot, x_0)$ and then
\begin{equation} \label{1759}
G_{\epsilon,x_0} \to G_\lambda (\cdot, x_0) \hbox{ in }W_0^{1,q}(\Omega)
\end{equation} 
as $\epsilon \to 0$,  for all $1\leq q< \bar{q}$.

\medskip \noindent To extend \eqref{1630} to the case $\lambda\not=0$, observe that \eqref{Gepsilonproblem_phdthesis} and $-\Delta_p \Gamma_{\epsilon,x_0}=f_{\epsilon,x_0}$ in $\Omega$ imply $\|H_{\epsilon,x_0} \|_\infty \leq C$ for all $\epsilon>0$ thanks to Proposition  \ref{lemmaGj} in view of $N<2p$ when $\lambda \not=0$.  Since $f=\lambda G_{\epsilon,x_0}^{p-1}$ is uniformly bounded in $L^{q_0}(\Omega)$ for some $q_0>\frac{N}{p}$ in view of $\frac{\bar q^*}{p-1}>\frac{N}{p}$ when $N<2p$ and
$$|\nabla \Gamma_{\epsilon,x_0}|=\frac{C_0(N-p)}{p-1} \frac{|x-x_0|^\frac1{p-1}}{(\epsilon^p + |x-x_0|^\frac{p}{p-1})^\frac{N}{p}} \leq M |\nabla \Gamma|(x,x_0),$$
we can apply \eqref{corollarylocalboundH3} in Theorem \ref{mainth} to $H_{\epsilon,x_0}$ as a solution to \eqref{Hepsilonproblem_phdthesis} by getting
\begin{equation} \label{uniformlocalboundedness}
\|H_{\epsilon,x_0} -H_\lambda(x_0,x_0)\|_{\infty, B_r(x_0)} \leq C \left(r^{-\frac{N}{p-1}} \| H_{\epsilon,x_0} -H_\lambda(x_0,x_0)\|_{p-1,B_{2r}(x_0)} + 
r^\frac{pq_0-N}{q_0(p-1)} \right)
\end{equation}
for all $\epsilon^{p-1}\leq r \leq  \frac{1}{4}\hbox{dist}(x_0,\partial \Omega)$.

\medskip \noindent By contradiction assume that \eqref{uniformconvergenceHepsilon} does not hold. Then there exist sequences $\epsilon_n \to 0$ and $x_n \in \Omega$ so that $|H_{\epsilon_n,x_0} (x_n) - H_\lambda(x_n,x_0)| \geq 2\delta>0$.  Since by elliptic estimates  \cite{dib,lieberman,serrin,tolksdorf} there holds
\begin{equation} \label{09422}
G_{\epsilon,x_0} \to G_\lambda (\cdot ,x_0) \hbox{ in }C^1_{\hbox{loc}} (\bar \Omega \setminus \{x_0 \})
\end{equation}
as $\epsilon \to 0$ in view of \eqref{Gepsilonproblem_phdthesis} and \eqref{1759}, we have that $\bar x=x_0$ and then 
\begin{equation} \label{1703}
|H_{\epsilon_n,x_0} (x_n) - H_\lambda(x_0,x_0)| \geq \delta
\end{equation}
thanks to  $H_\lambda (\cdot, x_0)\in C(\bar \Omega)$. Since by the Sobolev embedding theorem $H_{\epsilon,x_0} \to H_\lambda(\cdot, x_0)$ in $L^{p-1}(\Omega)$ as $\epsilon \to 0$ in view of \eqref{1759} and $\bar q^*>p-1$, 
we can insert \eqref{1703} into \eqref{uniformlocalboundedness} and get as $n \to +\infty$
\begin{equation} \label{uniformlocalboundednessbis}
\delta\leq  C \left(r^{-\frac{N}{p-1}} \| H_\lambda(\cdot,x_0) -H_\lambda(x_0,x_0)\|_{p-1,B_{2r}(x_0)} + 
r^\frac{pq_0-N}{q_0(p-1)} \right)
\end{equation}
for all $0< r\leq \frac{1}{4}\hbox{dist}(x_0,\partial \Omega)$.   Since 
\begin{equation*}
r^{-\frac{N}{p-1}}  \|H_\lambda (\cdot,x_0)- H_\lambda(x_0,x_0)\|_{p-1,B_{2r}(x_0)}  \leq C r^\alpha  \to 0
\end{equation*}
as $r\to 0$ thanks to \eqref{holdercontinuity},  estimate \eqref{uniformlocalboundednessbis} leads to a contradiction  and the proof is complete. 
\end{proof}
As a by-product we have the following useful result.
\begin{corollary} \label{corollaryexpansionPUepsilon}
Let $2\leq p \leq N$ and assume $N<2p$ if $\lambda \not=0$. Then the expansion
\begin{equation} \label{expansionPUepsilon}
PU_{\epsilon,x_0}=U_{\epsilon,x_0} + \frac{C_1}{C_0} \epsilon^\frac{N-p}{p}H_\lambda(\cdot,x_0) + o\bigl(\epsilon^\frac{N-p}{p}\bigr)
\end{equation}
does hold uniformly in $\Omega$ as $\epsilon \to 0$. 
\end{corollary}

\section{Energy expansions and Pohozaev identities}
We are concerned with the discussion of implication $(i)$ $\Rightarrow$ $(ii)$ in Theorem \ref{existenceBNproblem}, whereas the proof of $(ii)$ $\Rightarrow$ $(iii)$ in Theorem \ref{existenceBNproblem} is rather classical and can be found in \cite{gueddaveron2}. 

\medskip \noindent Let $0<\lambda<\lambda_1$ and $x_0 \in \Omega$ so that $H_\lambda(x_0,x_0)>0$. In order to show $S_\lambda<S_0$ let us expand $Q_\lambda(PU_{\epsilon,x_0})$ for $\epsilon>0$ small. Since $PU_{\epsilon,x_0}$ solves \eqref{PUepsilonproblem_phdthesis}, we have that
\begin{eqnarray} \label{expansionnum}
\int_\Omega |\nabla PU_{\epsilon,x_0}|^p - \lambda \int_\Omega (PU_{\epsilon,x_0})^p
& =& \int_\Omega U_{\epsilon,x_0}^{p^*-1}  PU_{\epsilon,x_0}  \nonumber\\
&= & \int_\Omega  U_{\epsilon,x_0}^{p^*} + \frac{C_1}{C_0} \epsilon^\frac{N-p}{p} \int_\Omega U_{\epsilon,x_0}^{p^*-1}\left[H_\lambda(x,x_0)+o(1)\right]
\end{eqnarray}
as $\epsilon \to 0$ in view of \eqref{expansionPUepsilon}. Given $\Omega_\epsilon = \frac{\Omega-x_0}{\epsilon^{p-1}}$ observe that
\begin{equation} \label{1918}
 \int_\Omega  U_{\epsilon,x_0}^{p^*} =\int_{\Omega_\epsilon} U_1^{p^*} = \int_{\R^N} U_1^{p^*} + O(\epsilon^N)
\end{equation}
and
\begin{eqnarray} \label{1919}
\int_\Omega U_{\epsilon,x_0}^{p^*-1}[H_\lambda(x,x_0)+o(1)] &=& \int_\Omega U_{\epsilon,x_0}^{p^*-1}[H_\lambda(x_0,x_0)+O(|x-x_0|^\alpha)+o(1)] \nonumber \\
&=&
\epsilon^{\frac{(N-p)(p-1)}{p}} \int_{\Omega_\epsilon} U_1^{p^*-1}[H_\lambda(x_0,x_0)+O(\epsilon^{\alpha(p-1)}|y|^\alpha)+o(1)] \nonumber \\
&=&\epsilon^{\frac{(N-p)(p-1)}{p}} H_\lambda(x_0,x_0) \int_{\mathbb{R}^N} U_1^{p^*-1}+o(\epsilon^{\frac{(N-p)(p-1)}{p}} )
\end{eqnarray}
in view of \eqref{holdercontinuity} and $\int_{\R^n} U_1^{p^*-1} |y|^\alpha < + \infty$. Inserting \eqref{1918}-\eqref{1919} into \eqref{expansionnum} we deduce
\begin{equation} \label{numeratoreespansione}
\int_\Omega |\nabla PU_{\epsilon,x_0}|^p - \lambda \int_\Omega (PU_{\epsilon,x_0})^p =  \int_{\R^N} U_1^{p^*} +  \epsilon^{N-p}  \frac{C_1}{C_0}  H_\lambda(x_0,x_0) \int_{\R^N} U_1^{p^*-1}+o(\epsilon^{N-p}).
\end{equation}
By the Taylor expansion
$$(PU_{\epsilon,x_0})^{p^*}= U_{\epsilon,x_0}^{p^*} +  \epsilon^\frac{N-p}{p}  \frac{C_1}{C_0} p^*  U_{\epsilon,x_0}^{p^*-1} [ H_\lambda(x,x_0) + o(1)] 
+O(\epsilon^{2\frac{N-p}{p}} U_{\epsilon,x_0}^{p^*-2} +\epsilon^N)$$
in view of \eqref{expansionPUepsilon} and $\|H_\lambda(\cdot,x_0) \|_\infty <+\infty$, we obtain 
\begin{equation} \label{denominatoreespansione}
\int_\Omega (PU_{\epsilon,x_0})^{p^*} = \int_{\mathbb{R}^N} U_1^{p^*} + \epsilon^{N-p}  \frac{C_1}{C_0} p^* H_\lambda(x_0,x_0) \int_{\mathbb{R}^N} U_1^{p^*-1} +o(\epsilon^{N-p})
\end{equation}
thanks to \eqref{1918}-\eqref{1919} and 
\begin{equation*}
\int_\Omega U_{\epsilon,x_0}^{p^*-2} = \epsilon^{2\frac{(N-p)(p-1)}{p}} \int_{\Omega_\epsilon} U_1^{p^*-2} = O(\epsilon^{2\frac{(N-p)(p-1)}{p}})
\end{equation*}
for $N<2p$. Expansions \eqref{numeratoreespansione}-\eqref{denominatoreespansione} now yield
$$Q_\lambda(PU_{\epsilon,x_0})= S_0- (p-1) S_0^\frac{p-N}{p} ( \int_{\mathbb{R}^N} U_1^{p^*-1}) \frac{C_1}{C_0} \epsilon^{N-p} H_\lambda(x_0,x_0) + o(\epsilon^{N-p})$$
in view of \eqref{equationbubble_phdthesis} and 
\begin{equation*}
S_0= \frac{ \int_{\R^N} |\nabla U_1|^p}{(\int_{\R^N} U_1^{p^*})^\frac{p}{p^*}}=
(\int_{\mathbb{R}^N} U_1^{p^*})^{\frac{p}{N}}.\end{equation*}
Then, for $\epsilon>0$ small we obtain that $S_\lambda < S_0$ thanks to $H_\lambda(x_0,x_0)>0$.

\medskip \noindent As already discussed in the Introduction,  a fundamental tool is represented by the Pohozaev identity.  Derived \cite{DFSV} for autonomous PDE's involving the $p-$Laplace operator, it extends to the non-autonomous case and writes, in the situation of our interest, as follows: if $u \in C^{1,\alpha}(\bar D)$ solves $-\Delta_p u=\lambda u^{p-1}+cu^{p^*-1}+f$ in $D$ for $f \in C^1(\bar D)$ and $c \in \{0,1\}$, given $x_0 \in \mathbb{R}^N$ there holds
\begin{eqnarray}\label{Pohoz1}
\int_D [N H-f \langle x-x_0,\nabla u \rangle-\frac{N-p}{p}|\nabla u|^p]
=\int_{\partial D} \langle x-x_0, -\frac{|\nabla u|^p}{p} \nu +|\nabla u|^{p-2} \partial_\nu u \nabla u +H \nu \rangle 
\end{eqnarray}
with $H(u)= \frac{\lambda}{p}u^p+ \frac{c}{p^*}u^{p^*}$ and
\begin{eqnarray}\label{Pohoz2}
\int_D |\nabla u|^p=\int_D [\lambda u^p+cu^{p^*}+fu]+\int_{\partial D} u |\nabla u|^{p-2} \partial_\nu u.
\end{eqnarray}
An integral identity of Pohozaev type  for $G_\lambda(\cdot, x_0)$ like \eqref{pohozaevGepsilonlimite} below is of fundamental importance since $H_\lambda(x_0,x_0)$ appears as a sort of residue.  In the semi-linear case such identity \eqref{pohozaevGepsilonlimite} holds in the limit of \eqref{Pohoz1}-\eqref{Pohoz2} on $B_\delta(x_0)$ as $\delta \to 0$ thanks to $\nabla H_\lambda(\cdot, x_0) \in L^\infty(\Omega)$, a property far from being obvious in the quasi-linear context where just integral bounds on $\nabla H_\lambda(\cdot, x_0)$ like \eqref{1656} are available. Instead, we can use the special approximating sequence $G_{\epsilon,x_0}$ to derive the following result.
\begin{proposition}
Let $2\leq p<N$ and assume $N<2p$ if $\lambda \not=0$. Given $x_0 \in \Omega$, $0<\delta<\hbox{dist }(x_0,\partial \Omega)$ and $\lambda<\lambda_1$,  there holds
\begin{eqnarray} \label{pohozaevGepsilonlimite}
C_0 H_{\lambda}(x_0,x_0) &=& \lambda \int_{B_\delta(x_0)} G^p_\lambda(x,x_0) dx+ 
\int_{\partial B_\delta(x_0)} \left( \frac{\delta}{p}|\nabla G_\lambda(x,x_0)|^p-\delta |\nabla G_\lambda(x,x_0)|^{p-2} (\partial_\nu G_\lambda(x,x_0))^2  \right. \nonumber \\
&& \left.   -  \frac{\lambda \delta}{p} G_\lambda^p(x,x_0)-\frac{N-p}{p} G_\lambda(x,x_0) |\nabla G_\lambda(x,x_0)|^{p-2} \partial_\nu G_\lambda(x,x_0) \right)d\sigma(x) 
\end{eqnarray}
for some $C_0>0$.
\end{proposition}
\begin{proof}
Since by elliptic regularity theory \cite {dib,lieberman,serrin,tolksdorf} $G_{\epsilon,x_0} \in C^{1,\alpha}(\bar \Omega)$ for some $\alpha \in (0,1)$ in view of \eqref{Gepsilonproblem_phdthesis}, we can apply the Pohozaev identity \eqref{Pohoz1} to $G_{\epsilon,x_0} $ with $c=0$ and $f=f_{\epsilon,x_0}$ on $D=B_\delta(x_0) \subset \Omega$ to get
\begin{eqnarray} \label{pohozaevGepsilon11}
&&\int_{\partial B_\delta(x_0)} \left(- \frac{\delta}{p}|\nabla G_{\epsilon,x_0}|^p+\delta |\nabla G_{\epsilon,x_0}|^{p-2}  (\partial_\nu G_{\epsilon,x_0})^2+ \frac{\lambda \delta }{p} G_{\epsilon,x_0}^p +
\frac{N-p}{p}G_{\epsilon,x_0} |\nabla G_{\epsilon,x_0}|^{p-2} \partial_\nu G_{\epsilon,x_0}   \right) \nonumber \\
&&=\int_{B_\delta(x_0)} \left( \lambda G_{\epsilon,x_0}^p  - \frac{N-p}{p} f_{\epsilon,x_0}  G_{\epsilon,x_0} 
- f_{\epsilon,x_0} \langle x-x_0, \nabla G_{\epsilon,x_0} \rangle \right)
\end{eqnarray}
in view of \eqref{Pohoz2}.  The approximating sequence $G_{\epsilon,x_0}$ has the key property that $\nabla G_{\epsilon,x_0}$ and $f_{\epsilon,x_0}$ are at main order multiples of $\nabla U_{\epsilon,x_0}$ and $U_{\epsilon,x_0}^{p^*-1}$, respectively, in such a way that $f_{\epsilon,x_0} \nabla G_{\epsilon,x_0}$ allows for a further integration by parts of the R.H.S. in \eqref{pohozaevGepsilon11}. The function $H_{\epsilon,x_0}$ appears in the remaining lower-order terms and explains why in the limit $\epsilon \to 0$ an additional term containing $H_\lambda(x_0,x_0)$ will appear in  \eqref{pohozaevGepsilonlimite}. The identity
\begin{eqnarray*}
&& \int_{B_\delta(x_0)}  U_{\epsilon,x_0}^{p^*-2} G_{\epsilon,x_0} \langle x-x_0, \nabla U_{\epsilon,x_0} \rangle = \int_{B_\delta(x_0)} U_{\epsilon,x_0}^{p^*-1}   \langle x-x_0, \nabla G_{\epsilon,x_0}-\nabla H_{\epsilon,x_0} +
H_{\epsilon,x_0}  \frac{\nabla U_{\epsilon,x_0}}{U_{\epsilon,x_0}} 
\rangle \\
&&= \int_{B_\delta(x_0)} U_{\epsilon,x_0}^{p^*-1} \langle x-x_0, \nabla G_{\epsilon,x_0} 
+p^* H_{\epsilon,x_0} \frac{\nabla U_{\epsilon,x_0}}{U_{\epsilon,x_0}} \rangle - \delta \int_{\partial B_\delta(x_0)} U_{\epsilon,x_0}^{p^*-1}H_{\epsilon,x_0} +N \int_{B_\delta(x_0)} U_{\epsilon,x_0}^{p^*-1} H_{\epsilon,x_0} 
\end{eqnarray*}
does hold thanks to $G_{\epsilon,x_0}=\Gamma_{\epsilon,x_0}+H_{\epsilon,x_0}$ and $\Gamma_{\epsilon,x_0} \nabla U_{\epsilon,x_0}=U_{\epsilon,x_0} (\nabla G_{\epsilon,x_0}-\nabla H_{\epsilon,x_0})$, which inserted into 
\begin{eqnarray*}
\int_{B_\delta(x_0)} U_{\epsilon,x_0}^{p^*-1} \langle x-x_0, \nabla G_{\epsilon,x_0} \rangle    &=&
\delta \int_{\partial B_\delta(x_0)} U_{\epsilon,x_0}^{p^*-1} G_{\epsilon,x_0}  - (p^*-1) \int_{B_\delta(x_0)}   U_{\epsilon,x_0}^{p^*-2} G_{\epsilon,x_0} \langle x-x_0, \nabla U_{\epsilon ,x_0} \rangle \\
&& - N \int_{B_\delta(x_0)}  U_{\epsilon,x_0}^{p^*-1} G_{\epsilon ,x_0}
\end{eqnarray*}
leads to
\begin{eqnarray} \label{1985}
\int_{B_\delta(x_0)} f_{\epsilon,x_0} \langle x-x_0, \nabla G_{\epsilon,x_0} \rangle    &=&
 - (p^*-1) \int_{B_\delta(x_0)} f_{\epsilon,x_0} H_{\epsilon,x_0} \Big[\langle x-x_0 ,\frac{ \nabla U_{\epsilon,x_0}}{U_{\epsilon,x_0}} \rangle+\frac{N-p}{p} \Big] \nonumber \\
&&- \frac{N-p}{p} \int_{B_\delta(x_0)}  f_{\epsilon,x_0} G_{\epsilon,x_0} +o_\epsilon(1)
\end{eqnarray}
as $\epsilon \to 0$ in view of \eqref{11011}-\eqref{01719} and \eqref{uniformconvergenceHepsilon}. Since there holds
\begin{eqnarray*}
&& \frac{p(p-1)}{N-p} C_0^{1-p} C_1^{-\frac{p^2}{N-p}} \int_{B_\delta(x_0)} f_{\epsilon,x_0} H_{\epsilon,x_0} \Big[\langle x-x_0 ,\frac{ \nabla U_{\epsilon,x_0}}{U_{\epsilon,x_0}} \rangle+\frac{N-p}{p} \Big]\\
&& =
  \epsilon^p \int_{B_\delta(x_0)}  H_{\epsilon,x_0}  \frac{(p-1)\epsilon^p -|x-x_0|^\frac{p}{p-1}}{(\epsilon^p + |x-x_0|^\frac{p}{p-1})^{N+2-\frac{N}{p}}} =   \int_{B_{\frac{\delta}{\epsilon^{p-1}}}(0)}  H_{\epsilon,x_0}(\epsilon^{p-1}y+x_0)  \frac{(p-1) -|y|^\frac{p}{p-1}}{(1 + |y|^\frac{p}{p-1})^{N+2-\frac{N}{p}}} \\
& & \to   \int_{\R^N}  \frac{(p-1) -|y|^\frac{p}{p-1}}{(1 + |y|^\frac{p}{p-1})^{N+2-\frac{N}{p}}}   H_\lambda(x_0,x_0)
\end{eqnarray*}
as $\epsilon \to 0$ in view of \eqref{propositionuniformconvergenceHepsilon}, \eqref{uniformconvergenceHepsilon} and the Lebesgue convergence Theorem, we can insert \eqref{1985} into \eqref{pohozaevGepsilon11} and as $\epsilon \to 0$ get the validity of
\begin{eqnarray*} 
C_0 H_\lambda(x_0,x_0) &=& \int_{B_\delta(x_0)} \lambda G_\lambda(x,x_0)^p dx +  \int_{\partial B_\delta(x_0)} \left( \frac{\delta}{p} |\nabla G_\lambda(x,x_0)|^p-\delta |\nabla G_\lambda (x,x_0)|^{p-2} (\partial_\nu G_\lambda (x,x_0))^2   \right. \\
&& \left.  -\frac{\lambda \delta}{p} G_\lambda^p (x,x_0) -\frac{N-p}{p}G_\lambda(x,x_0) |\nabla G_\lambda(x,x_0)|^{p-2} \partial_\nu G_\lambda(x,x_0) \right) d \sigma(x)
\end{eqnarray*}
in view of \eqref{09422} and $\displaystyle \lim_{\epsilon \to 0} G_{\epsilon,x_0}=G_\lambda(\cdot,x_0)$ in $L^p(\Omega)$ if $\lambda \not= 0$, as it follows by \eqref{1759} and $\bar q^*>p$ thanks to $N<2p\leq p^2$, where
$$ C_0=(p^*-1) \frac{N-p}{p(p-1)} C_0^{p-1} C_1^{\frac{p^2}{N-p}} \int_{\R^N}  \frac{|y|^\frac{p}{p-1}-(p-1) }{(1 + |y|^\frac{p}{p-1})^{N+2-\frac{N}{p}}}.$$
Concerning the sign of the constant $C_0$, observe that
\begin{equation*}
\begin{split}
\int_{\R^N} \frac{|y|^\frac{p}{p-1}}{(1+|y|^\frac{p}{p-1})^{N+2-\frac{N}{p}}} &= - \frac{p-1}{pN+p-N} \int_{\R^N} \langle y, \nabla (1+|y|^\frac{p}{p-1})^{
\frac{N}{p}-N-1} \rangle \\
&=\frac{N(p-1)}{pN+p-N} \int_{\R^N} (1+|y|^\frac{p}{p-1})^{\frac{N}{p}-N-1}
\end{split}
\end{equation*}
and then
$$\int_{\R^N} \frac{|y|^\frac{p}{p-1}}{(1+|y|^\frac{p}{p-1})^{N+2-\frac{N}{p}}}
= \frac{N(p-1)}{p} \int_{\R^N} \frac{1}{(1+|y|^\frac{p}{p-1})^{N+2-\frac{N}{p}}},$$
which implies $C_0>0$ in view of
$$\int_{\R^N} \frac{|y|^\frac{p}{p-1}-(p-1)}{(1+|y|^\frac{p}{p-1})^{N+2-\frac{N}{p}}}= \frac{(N-p)(p-1)}{p}  \int_{\R^N} (1+|y|^\frac{p}{p-1})^{\frac{N}{p}-N-2} >0.$$
The proof of \eqref{pohozaevGepsilonlimite} is complete.
\end{proof}

\section{The blow-up approach}
Following \cite{Dru} let us introduce the following blow-up procedure. Letting $\lambda_n = \lambda_* + \frac{1}{n}$, we have that $S_{\lambda_n} <S_0=S_{\lambda_*}$ and then $S_{\lambda_n}$ is achieved by a nonnegative $u_n \in W_0^{1,p}(\Omega)$ which, up to a normalization, satisfies
\begin{equation} \label{unproblemBN}
- \Delta_p u_n  = \lambda_n u_n^{p-1}+u_n^{p^*-1} \text{ in } \Omega, \quad \int_\Omega u_n^{p^*}=S_{\lambda_n}^{\frac{N}{p}}.
\end{equation}
Since $\lambda_*<\lambda_1$, by \eqref{unproblemBN} the sequence $u_n$ is uniformly bounded in $W_0^{1,p}(\Omega)$ and then, up to a subsequence, $u_n \rightharpoonup u_0 \geq 0$ in $W_0^{1,p}(\Omega)$ and a.e. in $\Omega$ as $n \to +\infty$. Since 
\begin{equation*}
Q_{\lambda_n}(u) = Q_{\lambda_*}(u) - \frac{1}{n}\frac{ \| u_n\|_p^p}{\| u_n\|_{p^*}^p} \geq S_0 -\frac{C}{n}
\end{equation*}
for some $C>0$ thanks to the H\"{o}lder's inequality, we deduce that 
\begin{equation} \label{9683}
\lim_{n \to +\infty} S_{\lambda_n} = S_0.
\end{equation}
By letting $n \to +\infty$ in \eqref{unproblemBN} we deduce that $u_0 \in W_0^{1,p}(\Omega)$ solves
$$- \Delta_p u_0= \lambda_* u_0^{p-1} +u_0^{p^*-1}  \text{ in } \Omega,\quad  \int_\Omega u_0^{p^*} \leq S_0^{\frac{N}{p}},$$
thanks to $u_n \to u_0$ a.e. in $\Omega$ as $n \to +\infty$ and the Fatou convergence Theorem, and then 
\begin{equation*}
S_0 \leq  Q_{\lambda_*}(u_0) =(\int_\Omega u_0^{p^*})^{\frac{p}{N}} \leq S_0
\end{equation*}
if $u_0 \not=0$. Since $S_{\lambda_*}=S_0$ would be achieved by $u_0$ if $u_0 \not=0$,  assumption \eqref{hypothesisSlambdastar} is crucial to guarantee $u_0=0$ and then
\begin{equation} \label{weakconvergenceunto0}
u_n \rightharpoonup 0 \text{ in } W_0^{1,p}(\Omega), \quad u_n \to 0 \hbox{ in }L^q(\Omega) \hbox{ for }1\leq q<p^* \hbox{ and a.e. in }\Omega
\end{equation}
in view of the Sobolev embedding Theorem. Since by elliptic regularity theory \cite{dib,lieberman,serrin,tolksdorf} and the strong maximum principle \cite{vazquez} $0<u_n \in C^{1,\alpha}(\bar \Omega)$ for some $\alpha \in (0,1)$, we can start a blow-up approach to describe the behavior of $u_n$ since $\|u_n\|_\infty \to +\infty$ as $n\to +\infty$, as it follows by \eqref{weakconvergenceunto0} and $\int_\Omega u_n^{p^*}=S_{\lambda_n}^{\frac{N}{p}} \to S_0^{\frac{N}{p}}$ as $n \to +\infty$.

\medskip \noindent Letting $x_n \in \Omega$ so that $u_n(x_n)=\displaystyle \max_\Omega u_n$, define the blow-up speed as $\mu_n = [ u_n(x_n) ]^{-\frac{p}{N-p}} \to 0$ as $n \to +\infty$ and the blow-up profile 
\begin{equation} \label{18571}
U _n(y)= \mu_n^\frac{N-p}{p}u_n(\mu_n y + x_n), \quad y \in \Omega_n =  \frac{\Omega-x_n}{\mu_n},
\end{equation} 
which satisfies 
\begin{equation} \label{problemUnBN}
- \Delta_p U_n   = \lambda_n \mu_n^p U_n^{p-1}+U_n^{p^*-1}  \text{ in } \Omega_n, \quad U_n =0 \text{ on } \partial \Omega_n, \quad 0 < U_n \leq U_n(0)=1 \text{ in } \Omega_n,
\end{equation}
and
$$\sup_{n \in \mathbb{N}} \left[ \int_{\Omega_n} |\nabla U_n|^p+\int_{\Omega_n} U_n^{p^*}\right]<+\infty.$$
Since $U_n$ is uniformly bounded in $C^{1, \alpha}(A \cap \Omega_n)$ for all $A \subset \subset \R^N$ by elliptic estimates \cite{dib,lieberman,serrin,tolksdorf}, we get that, up to a subsequence, $U_n \to U$  in $C^1_{\text{loc}}(\bar{\Omega}_\infty)$, where $\Omega_\infty$ is an half-space with $\text{dist}(0, \partial \Omega_\infty)=L \in (0,\infty]$ in view of $1=U_n(0)-U_n(y) \leq C |y|$ for $y \in B_2(0) \cap \partial \Omega_n$ and $U \in D^{1,p}(\Omega_\infty)$ solves 
\begin{equation*}
- \Delta_p U = U^{p^*-1}  \text{ in } \Omega_\infty, \quad
U=0 \text{ on } \partial \Omega_\infty, \quad 0 < U \leq U(0) =1 \text{ in } \Omega_\infty.
\end{equation*}
Since $L < + \infty$ would provide $U \in D_0^{1,p}(\Omega_\infty)$, by \cite{mercuriwillem} one would get $U=0$, in contradiction with $U(0)=1$. Since
\begin{equation} \label{rapportodistanzamun}
\lim_{n \to + \infty} \frac{\text{dist}(x_n, \partial \Omega)}{\mu_n} = \lim_{n \to + \infty} \text{dist}(0, \partial \Omega_n) =+ \infty,
\end{equation}
by \cite{sciunzi,Sci,vetois} we have that $U$ coincides with $U_\infty=(1 + \Lambda |y|^\frac{p}{p-1})^{-\frac{N-p}{p}}$, $\Lambda=C_1^{-\frac{p^2}{(N-p)(p-1)}}$ (by \eqref{definitionbubble_phdthesis} with $x_0=0$ and $\epsilon=C_1^\frac{p}{(N-p)(p-1)}$ to have $U_\infty(0)=1$). Since
\begin{equation} \label{47596}
U_n(y)=\mu_n^\frac{N-p}{p}u_n(\mu_n y + x_n) \to (1 + \Lambda |y|^\frac{p}{p-1})^{-\frac{N-p}{p}} \hbox{ uniformly in } B_R(0) 
\end{equation}
as $n \to +\infty$ for all $R>0$, in particular there holds 
\begin{equation} \label{limitenormapstarRn}
\lim_{R \to + \infty} \lim_{n \to + \infty} \int_{B_{ R \mu_n}(x_n)} u_n^{p^*} =\int_{\R^N} U_\infty^{p^*}=S_0^{\frac{N}{p}}.
\end{equation}
Contained in \eqref{unproblemBN}-\eqref{9683}, the energy information $\displaystyle \lim_{n \to +\infty} \int_\Omega u_n^{p^*}=S_0^{\frac{N}{p}}$ combines with \eqref{limitenormapstarRn} to give
\begin{equation}\label{13578}
\lim_{R \to + \infty} \lim_{n \to + \infty} \int_{\Omega \setminus B_{ R \mu_n}(x_n)} u_n^{p^*} =0,
\end{equation}
a property which will simplify the blow-up description of $u_n$. Up to a subsequence, let us assume $x_n \to x_0 \in \bar \Omega$ as $n \to +\infty$.

\medskip \noindent The proof of the implication $(iii)$ $\Rightarrow$ $(i)$ in Theorem \ref{existenceBNproblem} proceeds through the $5$ steps that will be developed below. The main technical point is to establish a comparison  between $u_n$ and the bubble
$$U_n(x)= \frac{ \mu_n^\frac{N-p}{p(p-1)} }{(\mu_n^{\frac{p}{p-1}} + \Lambda   |x-x_n|^\frac{p}{p-1})^\frac{N-p}{p}}$$
in the form $u_n \leq C U_n$ in $\Omega$, no matter $x_n$ tends to $\partial \Omega$ or not. Thanks to such a fundamental estimate, we will first apply some Pohoazev identity in the whole $\Omega_n$ to exclude the boundary blow-up $d_n=\hbox{dist }(x_n,\partial \Omega) \to 0$ as $n \to +\infty$. In the interior case, still by a Pohozaev identity on $B_\delta(x_n)$ as $n \to +\infty$ and $\delta \to 0$, we will obtain an information on the limiting blow-up point $x_0=\displaystyle \lim_{n \to +\infty} x_n \in \Omega$  in the form $H_{\lambda_*}(x_0,x_0)=0$ and then the property $H_\lambda(x_0,x_0)>H_{\lambda_*}(x_0,x_0)=0$ for $\lambda>\lambda_*$ will follow by the monotonicity of $H_\lambda(x_0, x_0)$. 

\medskip \noindent \emph{Step 1.} There holds $u_n \to 0$ in $C_{\text{loc}}(\bar{\Omega} \setminus \{x_0\})$ as $n\to+\infty$, where $x_0=\displaystyle \lim_{n \to +\infty} x_n \in \bar{\Omega}$.

\medskip \noindent First observe that 
\begin{equation} \label{12537}
u_n \to 0 \hbox{ in }L^{p^*}_{\text{loc}}(\bar \Omega \setminus \{x_0\})
\end{equation} 
as $n \to +\infty$ in view of \eqref{13578} and we are then concerned with establishing the uniform convergence by  a Moser iterative argument. Given a compact set $K \subset \bar \Omega \setminus \{x_0\}$, consider $\eta \in C_0^\infty (\R^N \setminus \{x_0\})$  be a cut-off function with $0\leq \eta \leq 1$ and $\eta=1$ in $K$. Since $u_n=0$ on $\partial \Omega$, use $\eta^p u_n^\beta$, $\beta \geq 1$, as a test function in \eqref{unproblemBN} to get
$$\frac{\beta p^p}{(\beta-1+p)^p} \int_\Omega \eta^p | \nabla w_n|^p \leq \frac{p^p}{(\beta-1+p)^{p-1}} \int_\Omega \eta^{p-1} |\nabla \eta| w_n   |\nabla w_n|^{p-1} +
\int_\Omega \lambda_n  \eta^p w_n^p + \int_\Omega  \eta^p u_n^{p^*-p}w_n^p$$
in terms of $w_n= u_n^\frac{\beta-1+p}{p}$ and then by the Young inequality
\begin{eqnarray} \label{weakformulationstima4}
 \int_\Omega \eta^p | \nabla w_n|^p \leq C \beta^p \left( \int_\Omega |\nabla \eta|^p w_n^p +\int_\Omega \eta^p w_n^p + \int_\Omega  \eta^p u_n^{p^*-p}w_n^p \right)
\end{eqnarray}
for some $C>0$. Since by the H\"older inequality
\begin{equation*}
\int_\Omega \eta^p u_n^{p^*-p}w_n^p \leq C ( \int_{\Omega \cap supp \ \eta } u_n^{p^*} )^\frac{p}{N} \| \eta w_n\|_{p^*}^p=o(\| \eta w_n\|_{p^*}^p)
\end{equation*}
as $n \to +\infty$ in view of  \eqref{12537} and $\Omega \cap supp \ \eta  \subset \subset \bar \Omega \setminus \{x_0\}$, by \eqref{weakformulationstima4} and the Sobolev embedding Theorem we deduce that
$$\|\eta w_n\|_{p^*}^p  \leq C  \|w_n\|_p ^p=C \int_\Omega u_n^{\beta-1+p} \to 0 $$
for all $1\leq \beta<p^*-p+1$  in view of \eqref{weakconvergenceunto0} and then $u_n \to 0$ in $L^q(K)$ for all $1\leq q< \frac{N p^*}{N-p}$ as $n \to +\infty$. We have then established that
\begin{equation} \label{125377}
u_n \to 0 \hbox{ in }L^q_{\text{loc}}(\bar \Omega \setminus \{x_0\})
\end{equation} 
as $n \to +\infty$ for all $1\leq q< \frac{N p^*}{N-p}$.  Since $\frac{N}{p}(p^*-p)=p^* < \frac{N p^*}{N-p}$, observe that \eqref{125377} now provides that the R.H.S. in the equation \eqref{unproblemBN} can be written as $(\lambda_n+u_n^{p^*-p}) u_n^{p-1}$ with a bound on the coefficient $\lambda_n+u_n^{p^*-p}$  in $L^{q_0}_{\text{loc}}(\bar \Omega \setminus \{x_0\})$ for some $q_0 >\frac{N}{p}$. 
Given compact sets $K \subset \tilde K \subset \bar \Omega \setminus \{x_0\}$ with $\hbox{dist }(K,\partial \tilde K)>0$, by \cite{serrin} we have the estimate $\|u_n\|_{\infty,K} \leq C \|u_n\|_{p,\tilde K}$ and then $u_n \to 0 $ in $C(K)$ as $n \to +\infty$ in view of \eqref{125377} and $p< \frac{N p^*}{N-p}$. The convergence $u_n \to 0$ in $C_{\text{loc}}(\bar{\Omega} \setminus \{x_0\})$ has been then established as $n\to+\infty$.

\medskip \noindent \emph{Step 2.} The following pointwise estimates
\begin{equation} \label{estimate5un}
\lim_{n \to +\infty} \max_\Omega |x-x_n|^\frac{N-p}{p} u_n <\infty, \quad \lim_{R \to + \infty} \lim_{n \to + \infty} \max_{\Omega \setminus B_{R \mu_n}(x_n)} |x - x_n|^\frac{N-p}{p} u_n =0
\end{equation}
do hold. 

\medskip \noindent By contradiction and up to a subsequence, assume the existence of $y_n \in \Omega$ such that either
\begin{equation} \label{10411}
|x_n - y_n|^\frac{N-p}{p} u_n(y_n) = \max_\Omega  |x-x_n|^\frac{N-p}{p} u_n \to + \infty 
\end{equation}
as $n \to + \infty$ or
\begin{equation}\label{10412}
\max_\Omega  |x-x_n|^\frac{N-p}{p} u_n \leq C_0,\quad |x_n - y_n|^\frac{N-p}{p} u_n(y_n) = \max_{\Omega \setminus B_{R_n \mu_n} (x_n)} |x-x_n|^\frac{N-p}{p} u_n \geq \delta>0
\end{equation}
for some $R_n \to +\infty$ as $n \to + \infty$. Setting $\nu_n =[ u_n(y_n)]^{-\frac{p}{N-p}}$,  there hold $\frac{| x_n - y_n |}{\nu_n} \to + \infty$ in case \eqref{10411},  $\frac{| x_n - y_n |}{\nu_n} \in [\delta^\frac{p}{N-p},C_0^{\frac{p}{N-p}}]$ in case \eqref{10412} and $\nu_n \to 0$ as $n \to + \infty$, since $x_n-y_n \to 0$ as $n \to +\infty$ when \eqref{10412} holds thanks to Step 1. Up to a further subsequence, let us assume that $\frac{x_n-y_n}{\nu_n} \to p$ as $n \to +\infty$, where $p=+\infty$ in case \eqref{10411} and $p \in \mathbb{R}^N \setminus \{0\}$ in case \eqref{10412}. Since $( \frac{|x_n - y_n|}{\mu_n} )^\frac{N-p}{p} \geq  ( \frac{|x_n - y_n|}{\mu_n} )^\frac{N-p}{p} U_n (\frac{y_n- x_n}{\mu_n}) =  |x_n - y_n|^\frac{N-p}{p} u_n(y_n)$ in view of \eqref{problemUnBN}, where $U_n$ is given by \eqref{18571}, then $\frac{|x_n - y_n|}{\mu_n} \to + \infty$ as $n \to + \infty$ also in case \eqref{10411}. Setting $V_n(y)= \nu_n^\frac{N-p}{p} u_n(\nu_n y + y_n)$ for $y \in \tilde{\Omega}_n= \frac{ \Omega-y_n}{\nu_n}$, then $V_n(0)=1$ and in $\tilde \Omega_n$ there hold:
\begin{equation} \label{boundcase1}
V_n(y) \leq \nu_n^\frac{N-p}{p} |\nu_n y + y_n -x_n|^{-\frac{N-p}{p}} |x_n - y_n|^\frac{N-p}{p} u_n(y_n) = ( \frac{|x_n - y_n|}{|\nu_n y + y_n - x_n|} )^\frac{N-p}{p} \leq 2^\frac{N-p}{p}
\end{equation}
for $|y| \leq \frac{1}{2} \frac{|x_n - y_n|}{\nu_n}$ in case \eqref{10411} and
\begin{equation} \label{boundcase2}
|y- \frac{x_n-y_n}{\nu_n}|^\frac{N-p}{p} V_n(y)  = |\nu_n y + y_n - x_n|^\frac{N-p}{p} u_n (\nu_n y + y_n) \leq C_0
\end{equation}
in case \eqref{10412}. Since 
$$- \Delta_p V_n  = \lambda_n \nu_n^p V_n^{p-1} +V_n^{p^*-1}  \text{ in } \tilde \Omega_n, \quad V_n =0 \text{ on } \partial \tilde \Omega_n,$$
by \eqref{boundcase1}-\eqref{boundcase2} and standard elliptic estimates \cite{dib,lieberman,serrin,tolksdorf} we get that $V_n$ is uniformly bounded in $C^{1, \alpha}(A \cap \tilde{\Omega}_n)$ for all $A \subset \subset \R^N \setminus \{p\}$. Up to a subsequence, we have that $V_n \to V$  in $C^1_{\text{loc}}(\bar{\Omega}_\infty \setminus \{p\})$, where $\Omega_\infty$ is an half-space with $\text{dist}(0, \partial \Omega_\infty)=L$.
Since $p \not= 0$, there hold $B_\frac{|p|}{2}(0)\subset \subset \mathbb{R}^N \setminus \{p\}$ and $1= V_n(0)-V_n(y) \leq C |y|$ for $y \in B_\frac{|p|}{2}(0) \cap \partial \tilde{\Omega}_n$, leading to $L \in (0,\infty]$. Since $V \geq 0$ solves $-\Delta_p V = V^{p^*-1}$ in $\Omega_\infty$, by the strong maximum principle \cite{vazquez} we deduce that $V>0$ in $\Omega_\infty$ in view of $V(0)=1$ thanks to $0 \in \Omega_\infty$. Setting $M=\min \{L, |p|\}$, by  $\frac{|x_n - y_n|}{\mu_n} \to + \infty$ as $n \to + \infty$ we have that $B_{\frac{M}{2}\nu_n}(y_n) \subset \Omega  \setminus B_{ R \mu_n}(x_n)$ for all $R>0$ provided $n$ is sufficiently large (depending on $R$) and then
$$ \int_{\Omega \setminus B_{ R \mu_n}(x_n)} u_n^{p^*} \geq  \int_{B_{\frac{M}{2}\nu_n}(y_n) } u_n^{p^*}=
\int_{B_{\frac{M}{2}}(0) } V_n^{p^*} \to \int_{B_{\frac{M}{2}}(0) } V^{p^*} >0,$$
in contradiction with \eqref{13578}. The proof of \eqref{estimate5un} is complete.

\medskip \noindent \emph{Step 3.} There exists $C>0$ so that
\begin{equation} \label{eni11}
u_n \leq \frac{C \mu_n^\frac{N-p}{p(p-1)} }{(\mu_n^{\frac{p}{p-1}} + \Lambda   |x-x_n|^\frac{p}{p-1})^\frac{N-p}{p}}\quad \text{ in } \Omega
\end{equation}
does hold for all $n \in \N$.

\medskip \noindent Since \eqref{eni11} does already hold in $B_{R\mu_n}(x_n)$ for all $R>0$ thanks to \eqref{47596}, notice that \eqref{eni11} is equivalent to establish the estimate
\begin{equation} \label{stimafondamentaleun}
u_n \leq \frac{C \mu_n^{\frac{N-p}{p(p-1)}}}{|x-x_n|^{\frac{N-p}{p-1}}} \quad \text{ in } \Omega \setminus B_{R\mu_n}(x_n)
\end{equation}
for some $C, R>0$ and all $n \in N$. Let us first prove the following weaker form of \eqref{stimafondamentaleun}: given $0<\eta< \frac{N-p}{p(p-1)}$ there exist $C, R>0$ so that
\begin{equation} \label{73950bis}
u_n \leq  \frac{C \mu_n^{\frac{N-p}{p(p-1)}-\eta}}{|x - x_n|^{\frac{N-p}{p-1} - \eta}} \quad \hbox{ in }\Omega  \setminus B_{R \mu_n }(x_n)
\end{equation}
does hold for all $n \in N$. Since $|x - x_n|^{\eta-\frac{N-p}{p-1}}$ satisfies
\begin{equation*}
- \Delta_p |x - x_n|^{\eta-\frac{N-p}{p-1}}= \eta(p-1) (\frac{N-p}{p-1} - \eta)^{p-1}  |x-x_n|^{\eta(p-1)-N}, 
\end{equation*}
we have that $\Phi_n=C \frac{\mu_n^{\frac{N-p}{p(p-1)}-\eta} + M_n}{|x - x_n|^{\frac{N-p}{p-1} - \eta}}$, where $\rho,C>0$ and $M_n= \displaystyle \sup_{\Omega \cap \partial B_\rho(x_0) } u_n$, satisfies
\begin{eqnarray*}
-\Delta_p \Phi_n- ( \lambda_n + \frac{\delta}{|x-x_n|^p} ) \Phi_n^{p-1}&=&\left[ \eta(p-1) ( \frac{N-p}{p-1} - \eta)^{p-1}  - (\lambda_n |x-x_n|^p + \delta) \right] \frac{\Phi_n^{p-1}}{|x-x_n|^p} \\
&\geq& 0 \quad \hbox{ in } \Omega \cap B_\rho(x_0) \setminus \{x_n\} 
\end{eqnarray*}
provided $\rho$ and $\delta$ are sufficiently small (depending on $\eta$). Taking $R>0$ large so that $u_n^{p^*-p} \leq \frac{\delta}{|x-x_n|^p}$ in $\Omega \setminus B_{R\mu_n}(x_n)$  for all $n$ large thanks to \eqref{estimate5un}, we have that
$$-\Delta_p u_n - ( \lambda_n + \frac{\delta}{|x-x_n|^p}) u_n^{p-1} =(u_n^{p^*-p} -\frac{\delta}{|x-x_n|^p})u_n^{p-1} \leq 0 \quad \hbox{ in } \Omega \setminus B_{R\mu_n}(x_n).$$
By  \eqref{47596} on $\partial B_{R \mu_n }(x_n)$ it is easily seen that $u_n \leq \Phi_n$  on the boundary of $ \Omega \cap B_\rho(x_0) \setminus B_{R \mu_n }(x_n)$ for some $C>0$, and then by Proposition \ref{wcp} one deduces the validity of
\begin{equation} \label{73950}
u_n \leq C \frac{\mu_n^{\frac{N-p}{p(p-1)}-\eta} + M_n}{|x - x_n|^{\frac{N-p}{p-1} - \eta}} 
\end{equation}
in $\Omega \cap B_\rho(x_0) \setminus B_{R \mu_n }(x_n)$. Setting $A=\Omega \setminus B_\rho(x_0)$, observe that the function $v_n=\frac{u_n}{M_n}$ satisfies
\begin{equation} \label{00953}
- \Delta_p v_n - \lambda_n v_n^{p-1} = f_n \hbox{ in }\Omega, \quad v_n=0 \hbox{ on }\partial \Omega, \quad \sup_{\Omega \cap \partial B_\rho(x_0)} v_n=1,
\end{equation}
where $f_n= \frac{u_n^{p^*-1}}{M_n^{p-1}}=u_n^\frac{p^2}{N-p} v_n^{p-1}$. Letting $g_n$ be the $p-$harmonic function in $A$ so that $g_n=v_n$ on $\partial A$, observe that 
$\|g_n\|_\infty=1$ in view of $0\leq v_n \leq 1$ on $\partial A$. Since by Step $1$ there holds
$$a_n =\lambda_n +u_n^\frac{p^2}{N-p} \to \lambda_* \quad \hbox{ in }L^\infty (A)$$
as $n \to +\infty$ with $\lambda_*<\lambda_1(\Omega)<\lambda_1(A)$, by Proposition \ref{lemmaGj} we deduce that $\displaystyle \sup_{n\in \mathbb{N}} \|v_n\|_{p-1,A}<+\infty$ and then $\displaystyle \sup_{n\in \mathbb{N}} \|f_n\|_{1,A}<+\infty$ in view of Step $1$. Letting $w_n$ the solution of
\begin{equation*}
-\Delta_p w_n= f_n \text{ in } A, \quad w_n=0 \text{ on } \partial A,
\end{equation*} 
by Proposition \ref{lemmaGj} we also deduce that $\displaystyle \sup_{n\in \mathbb{N}} \|v_n-w_n\|_{\infty,A}<+\infty$ thanks to $N<2p$. Since by the Sobolev embedding Theorem $\displaystyle \sup_{n\in \mathbb{N}} \|w_n\|_{q,A}<+\infty$ for all $1\leq q<\bar q^*$ in view of Proposition \ref{lemmaGj} and $\displaystyle \sup_{n\in \mathbb{N}} \|f_n\|_{1,A}<+\infty$, similar estimates hold for $v_n$ and then $\displaystyle \sup_{n\in \mathbb{N}} \|f_n\|_{q_0,A}<+\infty$ for some $q_0>\frac{N}{p}$ in view of $N<2p$. By elliptic estimates \cite{serrin} we get that $\displaystyle \sup_{n\in \mathbb{N}}\|w_n\|_{\infty,A}<+\infty$ and in turn $\displaystyle \sup_{n\in \mathbb{N}} \|v_n\|_{\infty,A}<+\infty$, or equivalently
\begin{equation} \label{18463}
\sup_{\Omega \setminus B_\rho(x_0)} u_n \leq C \sup_{\Omega \cap \partial B_\rho(x_0) } u_n
\end{equation}
for some $C>0$. Thanks to \eqref{18463} one can extend the validity of \eqref{73950}  from $\Omega \cap B_\rho(x_0) \setminus B_{R \mu_n }(x_n)$ to $\Omega \setminus B_{R \mu_n }(x_n)$. In order to establish \eqref{73950bis}, we claim that $M_n$ in \eqref{73950} satisfies
\begin{equation} \label{7957}
M_n=o(\mu_n^{\frac{N-p}{p(p-1)} - \eta})
\end{equation}
for all $0<\eta <\frac{N-p}{p(p-1)}$.

\medskip \noindent Indeed,  by contradiction assume that there exist $0<\bar \eta <\frac{N-p}{p(p-1)}$ and a subsequence so that
\begin{equation} \label{1848}
\mu_n^{\frac{N-p}{p(p-1)} - \bar \eta} \leq CM_n
\end{equation}
for some $C>0$.  Since $v_n=O(|x - x_n|^{-\frac{N-p}{p-1} +\bar \eta} )$ uniformly in $\Omega \setminus B_{R \mu_n }(x_n)$ in view of \eqref{73950}  and \eqref{1848}, we have that $v_n$ and then $f_n=u_n^\frac{p^2}{N-p} v_n^{p-1}$ are uniformly bounded in $C_{\hbox{loc}}(\bar \Omega \setminus \{x_0\})$ and by elliptic estimates \cite{dib,lieberman,serrin,tolksdorf} $v_n \to v$ in $C^1_{\hbox{loc}}(\bar \Omega \setminus \{x_0\})$ as $n \to +\infty$, up to a further subsequence, where $v \not=0$ in view of $\displaystyle \sup_{\Omega \cap \partial B_\rho(x_0)} v=\displaystyle \lim_{n \to +\infty} \sup_{\Omega \cap \partial B_\rho(x_0)} v_n=1$. Moreover,  notice that $\displaystyle \lim_{n \to +\infty} \|f_n \|_1 = 0$ would imply $v_n \to v$ in $W_0^{1,q}(\Omega)$ for all $1\leq q<\bar{q}$ and in $L^s(\Omega)$ for all $1\leq s<\bar q^*$ as $n \to +\infty$ in view of Proposition \ref{lemmaGj},  where $v$ is a solution of
\begin{equation} \label{18468}
- \Delta_p v- \lambda_* v^{p-1} = 0 \quad \hbox{ in }\Omega. 
\end{equation}
Letting 
$$T_l(s)= \left\{ \begin{array}{ll} |s| &\hbox{if }|s|\leq l\\ \pm l &\hbox{if }\pm s>l \end{array}\right.$$
and using $T_l(v_n)\in W^{1,p}_0(\Omega)$ as a test function in \eqref{00953}, one would get
$$ \int_{\{|v_n| \leq l \} } |\nabla v_n|^p \leq \lambda_n  \int_\Omega v_n^p+l \|f_n\|_1 \to \lambda_* \int_\Omega v^{p} $$
as $n \to +\infty$ in view of $\bar q^*>p$ and then deduce 
$$ \int_\Omega |\nabla v|^p \leq \lambda_* \int_\Omega v^{p} <+\infty$$
as $l \to +\infty$. Since $v \in W^{1,p}_0(\Omega)$ solves \eqref{18468} with $\lambda_*<\lambda_1$, one would have $v=0$, in contradiction with $\displaystyle \sup_{\Omega \cap \partial B_\rho(x_0)} v=1$. Once
\begin{equation} \label{74558}
\liminf_{n\to +\infty}\|f_n \|_1>0
\end{equation}
has been established, by \eqref{18571}, \eqref{47596} and \eqref{73950}  observe that
\begin{equation} \label{47586}
\int_{B_{R\mu_n}(x_n)} f_n = \int_{B_{R\mu_n}(x_n)} \frac{u_n^{p^*-1}}{M_n^{p-1}}= \frac{\mu_n^\frac{N-p}{p}}{M_n^{p-1}} \int_{B_R(0)} U_n^{p^*-1} = O\left( \frac{\mu_n^\frac{N-p}{p}}{M_n^{p-1}}\right) 
\end{equation}
and
\begin{eqnarray} \label{82648}
\int_{\Omega \setminus B_{R\mu_n}(x_n)} f_n =  L_n^\frac{p^2}{N-p} (\frac{L_n}{M_n})^{p-1} O \left(  \mu_n^{(p^*-1)\eta-\frac{p}{p-1}} \log \frac{1}{\mu_n}+1\right) 
\end{eqnarray}
where $L_n=\mu_n^{\frac{N-p}{p(p-1)}-\eta} + M_n$. Setting $\eta_0=\frac{p}{(p-1)(p^*-1)}$, then \eqref{7957} necessarily holds for $\eta \in (\eta_0,\frac{N-p}{p(p-1)})$ since otherwise $L_n=O(M_n)$ and \eqref{47586}-\eqref{82648} would provide $\displaystyle \lim_{n \to +\infty} \|f_n\|=0$ along a subsequence thanks to $\displaystyle \lim_{n \to +\infty}M_n=0$, in contradiction with \eqref{74558}. Notice that \eqref{7957} holds for $\eta=\eta_0$ too,  since otherwise the conclusion $\displaystyle \lim_{n \to +\infty} \|f_n\|=0$ would follow as above thanks to $L_n=O(\mu_n^{\frac{N-p}{p(p-1)}-\eta})$ for $\eta \in (\eta_0,\frac{N-p}{p(p-1)})$.  Setting $\eta_k= ( \frac{p^2}{Np-N+p})^k \eta_0$, arguing as above \eqref{7957} can be established for $\eta \in [\eta_{k+1},\eta_k)$, $k \geq 0$, by using  the validity of \eqref{7957} for $\eta \in [\eta_k,\frac{N-p}{p-1})$ in view of the relation
$$(p^*-1) \eta_{k+1} -\frac{p}{p-1}+\frac{p^2}{N-p}\left[\frac{N-p}{p(p-1)}-\eta_k\right]=0.$$
Since $\frac{p^2}{Np-N+p}<1$ for $p<N$, we have that $\eta_k \to 0 $ as $k \to + \infty$ and then \eqref{7957} is proved for all $0<\eta <\frac{N-p}{p(p-1)}$,  in contradiction with \eqref{1848}.  Therefore, we have established \eqref{7957} and the validity of \eqref{73950bis} follows.

\medskip \noindent In order to establish \eqref{stimafondamentaleun}, let us repeat the previous argument for $v_n = \mu_n^{-\frac{N-p}{p(p-1)}} u_n$, where $v_n$ solves
\begin{equation} \label{164833}
- \Delta_p v_n - \lambda_n v_n^{p-1} = f_n \hbox{ in }\Omega, \quad v_n=0 \hbox{ on }\partial \Omega,
\end{equation}
with $f_n=\mu_n^{-\frac{N-p}{p}} u_n^{p^*-1}$. Notice that $f_n$ satisfies
\begin{equation} \label{72649}
f_n \leq \frac{C_0 \mu_n^{\frac{p}{p-1} - (p^*-1)\eta}}{ |x-x_n|^{N+\frac{p}{p-1}-(p^*-1)\eta }} \quad \hbox{ in }\Omega \setminus B_{R\mu_n}(x_n)
\end{equation}
for some $C_0>0$ in view of \eqref{73950bis} and then, by arguing as in \eqref{47586},
\begin{equation} \label{19374}
\int_{\Omega} f_n=O(1) 
+O\Big(\int_{\Omega \setminus B_{R\mu_n}(x_n)} \frac{ \mu_n^{\frac{p}{p-1} - (p^*-1)\eta}}{ |x-x_n|^{N+\frac{p}{p-1}-(p^*-1)\eta }} \Big)= O(1)
\end{equation}
for $0<\eta <\frac{p}{(p-1)(p^*-1)}=\frac{p(N-p)}{(p-1)(Np-N+p)}$.  Letting $h_n$ be the solution of 
\begin{equation*}
-\Delta_p h_n= f_n \text{ in } \Omega, \quad h_n=0 \text{ on } \partial \Omega,
\end{equation*}
by \eqref{19374} and Proposition \ref{lemmaGj} we deduce that $\displaystyle \sup_{n \in \mathbb{N}} \|v_n-h_n\|_\infty<+\infty$ thanks to $N<2p$, or equivalently
\begin{equation} \label{15957}
\|u_n -\mu_n^{\frac{N-p}{p(p-1)}} h_n\|_\infty=O(\mu_n^{\frac{N-p}{p(p-1)}}).
\end{equation}
For $\alpha>N$ the radial function  
$$W(y)=(\alpha-N)^{-\frac{1}{p-1}} \int_{|y|}^\infty \frac{(t^{\alpha-N}-1)^\frac{1}{p-1}}{t^\frac{\alpha-1}{p-1}}dt$$
is a positive and strictly decreasing solution of $-\Delta_p W=|y|^{-\alpha}$ in $\R^N\setminus B_1(0)$ so that
\begin{equation}\label{93749}
\lim_{|y|\to \infty} |y|^\frac{N-p}{p-1} W(y)=\frac{p-1}{N-p}(\alpha-N)^{-\frac{1}{p-1}}>0.
\end{equation}
Taking $0<\eta <\frac{p}{(p-1)(p^*-1)}$ to ensure $\alpha:=N+\frac{p}{p-1}-(p^*-1)\eta>N$, then $w_n(x)=\mu_n^{-\frac{N-p}{p-1}} W(\frac{x-x_n}{\mu_n})$ satisfies
$$- \Delta_p w_n=  \frac{\mu_n^{\frac{p}{p-1}-(p^*-1)\eta} }{ |x-x_n|^{N+\frac{p}{p-1}-(p^*-1)\eta}}\quad \hbox{in }\R^N\setminus B_1(x_n).$$
Since
$$h_n (x)= \mu_n^{-\frac{N-p}{p(p-1)}} u_n(x)+O(1)=  \mu_n^{-\frac{N-p}{p-1}} U_n( \frac{x-x_n}{\mu_n})+O(1) \leq
C_1 w_n(x)$$
for some $C_1>0$ and for all $x \in \partial B_{R\mu_n}(x_n)$ in view of \eqref{47596}, \eqref{15957} and $W(R)>0$, we have that $\Phi_n=C w_n$ satisfies
$$-\Delta_p \Phi_n \geq f_n \hbox{ in }\Omega \setminus B_{R\mu_n}(x_n),\quad \Phi_n \geq h_n \hbox{ on }\partial \Omega\cup \partial B_{R \mu_n}(x_n)$$ 
for $C=C_0^\frac{1}{p-1}+C_1$ thanks to \eqref{72649}, and then by weak comparison principle we deduce that
\begin{equation} \label{24944}
h_n \leq \Phi_n\leq \frac{C}{ |x-x_n|^\frac{N-p}{p-1} } \quad \hbox{ in }\Omega \setminus B_{R\mu_n}(x_n)
\end{equation}
for some $C>0$ in view of \eqref{93749}. Inserting \eqref{24944} into \eqref{15957} we finally deduce the validity of \eqref{stimafondamentaleun}

\medskip \noindent \emph{Step 4.} There holds $x_0 \notin \partial \Omega$

\medskip \noindent Assume by contradiction $x_0 \in \partial \Omega$ and set $\hat{x}=x_0 - \nu(x_0)$. Let us apply the Pohozaev identity \eqref{Pohoz1} to $u_n$ with $c=1$, $f=0$ and $x_0=\hat x$ on $D=\Omega$, together with \eqref{Pohoz2}, to get
\begin{equation} \label{3759}
\int_{\partial \Omega} |\nabla u_n|^p   \langle x - \hat{x}, \nu \rangle = \frac{p }{p-1} \lambda_n \int_\Omega u_n^p 
\end{equation}
in view of $u_n= 0 $  and $\nabla u_n= (\partial_\nu u_n) \nu$ on $\partial \Omega$. Since $v_n=\mu_n^{-\frac{N-p}{p(p-1)}}u_n$ solves
\eqref{164833} and $v_n,f_n$ are uniformly bounded in $C_{\hbox{loc}}(\bar \Omega \setminus \{x_0\})$ in view of \eqref{eni11} and \eqref{72649} with $\eta=0$, by elliptic estimates  \cite{dib,lieberman,serrin,tolksdorf} we deduce that $v_n $ is uniformly bounded in $C^1_{\hbox{loc}}(\bar \Omega \setminus \{x_0\})$. Fixing $\rho>0$ small so that $\langle x-\hat x, \nu(x)\rangle \geq \frac{1}{2}$ for all $x \in \partial \Omega \cap B_\rho(x_0)$, by \eqref{eni11}, \eqref{3759} and the $C^1-$bound on $v_n$ we have that
\begin{equation} \label{61783}
\int_{\partial \Omega \cap B_\rho(x_0)} |\nabla u_n|^p  =O(\lambda_n \int_\Omega u_n^p +\int_{\partial \Omega \setminus B_\rho(x_0)} |\nabla u_n|^p )
=O( \mu_n^{\frac{N-p}{p-1}})
\end{equation} 
since $\frac{p(N-p)}{p-1}<N$ thanks $N<2p\leq p^2$. Setting $d_n=\hbox{dist}(x_n,\partial \Omega)$ and $W_n(y)= d_n^\frac{N-p}{p} u_n(d_n y+x_n)$ for $y \in \Omega_n = \frac{ \Omega-x_n}{d_n}$, we have that $d_n \to 0$ and $\Omega_n \to \Omega_\infty$ as $n \to +\infty$ where $\Omega_\infty$ is an halfspace containing $0$ with $\hbox{dist}(0,\partial \Omega_\infty)=1$. Setting $\delta_n= \frac{\mu_n}{d_n} \to 0$ as $n \to + \infty$ in view of \eqref{rapportodistanzamun}, the function $G_n=\delta_n^{-\frac{N-p}{p(p-1)}} W_n =\mu_n^{-\frac{N-p}{p(p-1)}} d_n^\frac{N-p}{p-1} u_n(d_n y+x_n)\geq 0$ solves 
\begin{equation} \label{problemGn}
-\Delta_p G_n - \lambda_n d_n^p G_n^{p-1}= \tilde f_n \text{ in } \Omega_n,\quad G_n=0  \text{ on } \partial \Omega_n,
\end{equation}
with $\tilde f_n= \mu_n^{-\frac{N-p}{p}} d_n^N u_n^{p^*-1}(d_ny+x_n) =d_n^N f_n(d_ny+x_n)$ so that
\begin{equation} \label{tildeGnbound}
\tilde f_n \leq \frac{C \delta_n^\frac{p}{p-1}}{|y|^{N+\frac{p}{p-1}}},\: 
G_n \leq \frac{C}{|y|^\frac{N-p}{p-1}} \quad \hbox{ in }\Omega_n
\end{equation}
in view of \eqref{eni11} and \eqref{72649} with $\eta=0$. By \eqref{tildeGnbound} and elliptic estimates \cite{dib,lieberman,serrin,tolksdorf} we deduce that $G_n \to G$ in $C^1_{\hbox{loc}}(\bar \Omega_\infty \setminus \{0\})$ as $n \to +\infty$, where $G \geq 0$ does solve
$$-\Delta_p G= \left(\int_{\R^N} U^{p^*-1}\right) \delta_0 \text{ in } \Omega_\infty,\quad G=0  \text{ on } \partial \Omega_\infty,$$
in view of \eqref{problemGn} and
\begin{equation} \label{68494}
\lim_{n \to +\infty} \int_{B_\epsilon(0)} \tilde f_n=
\lim_{n \to +\infty} \mu_n^{-\frac{N-p}{p}} \int_{B_{\epsilon d_n}(x_n)} u_n^{p^*-1} =
\lim_{n \to +\infty}  \int_{B_{\epsilon \frac{d_n}{\mu_n}}(0)} U_n^{p^*-1}=\int_{\R^N} U^{p^*-1}
\end{equation}
for all $\epsilon>0$ in view of \eqref{rapportodistanzamun}-\eqref{47596} and \eqref{eni11}. By the strong maximum principle \cite{vazquez} we then have that $G>0$ in $\Omega_\infty$ and $\partial_\nu G<0$ on $\partial \Omega_\infty$. On the other hand, for any $R>0$ there holds
$$ \int_{\partial \Omega_n \cap B_R(0)} |\nabla G_n|^p= \mu_n^{-\frac{N-p}{p-1}} d_n^\frac{N-1}{p-1} \int_{\partial \Omega \cap B_{Rd_n}(x_n)} |\nabla u_n|^p =O( d_n^\frac{N-1}{p-1} )$$
in view of \eqref{61783} and then as $n \to +\infty$
$$\int_{\partial \Omega_\infty \cap B_R(0)} |\nabla G|^p=0.$$
We end up with the contradictory conclusion $\nabla G=0$ on $\partial \Omega_\infty$, and then $x_0 \notin \partial \Omega$.

\medskip \noindent \emph{Step 5.} There holds $H_{\lambda_*}(x_0,x_0)=0$.

\medskip \noindent Let us apply the Pohozaev identity \eqref{Pohoz1} to $u_n$ with $c=1$ and $f=0$ on $D=B_\delta(x_0) \subset \Omega$ and \eqref{Pohoz2} to get
\begin{eqnarray} \label{pohozaevweakunBdelta}
\lambda_n \int_{B_\delta(x_0)} u_n^p &+& \int_{\partial B_\delta(x_0)} \left( \frac{\delta}{p} |\nabla u_n|^p
-\delta |\nabla u_n|^{p-2} (\partial_\nu u_n)^2 - \frac{\lambda_n \delta}{p} u_n^p  -\frac{N-p}{p} u_n |\nabla u_n|^{p-2} \partial_\nu u_n\right) \nonumber \\
&-&\frac{N-p}{Np} \delta \int_{\partial B_\delta(x_0)}  u_n^{p^*}=0 . 
\end{eqnarray}
As in the previous Step, up to a subsequence, there holds $G_n=\mu_n^{-\frac{N-p}{p(p-1)}}u_n \to G$ in $C^1_{\hbox{loc}}(\bar \Omega \setminus \{x_0\})$ as $n \to +\infty$, where $G\geq 0$ satisfies
$$- \Delta_p G- \lambda_* G^{p-1} = \left(\int_{\R^N} U^{p^*-1}\right) \delta_{x_0} \hbox{ in }\Omega, \quad G=0 \hbox{ on }\partial \Omega,$$
as it follows by \eqref{problemGn} and \eqref{68494} with $d_n=1$. Arguing as in Proposition \ref{propositionuniformconvergenceHepsilon}, we can prove that $H=G-\Gamma$ satisfies \eqref{1656} and by Theorem \ref{mainth} it follows that $G=\left(\int_{\R^N} U^{p^*-1}\right) ^\frac{1}{p-1} G_{\lambda_*}(\cdot, x_0)$. Since $\mu_n^{-\frac{N-p}{p(p-1)}}u_n  \to \left(\int_{\R^N} U^{p^*-1}\right) ^\frac{1}{p-1} G_{\lambda_*}(\cdot,x_0)$ in $C^1_{\hbox{loc}}(\bar \Omega \setminus \{x_0\})$ as $n \to +\infty$, by letting $n \to +\infty$ in \eqref{pohozaevweakunBdelta} we finally get 
\begin{eqnarray*} 
&& \lambda_* \int_{B_\delta(x_0)} G^p_{\lambda_*}(x,x_0) dx+ \int_{\partial B_\delta(x_0)} \left( \frac{\delta}{p} |\nabla G_{\lambda_*}(x,x_0)|^p
-\delta |\nabla G_{\lambda_*}(x,x_0)|^{p-2} (\partial_\nu G_{\lambda_*}(x,x_0))^2  \right. \\
&&\left. - \frac{\lambda_* \delta}{p} G^p_{\lambda_*}(x,x_0) -\frac{N-p}{p} G_{\lambda_*}(x,x_0) |\nabla G_{\lambda_*}(x,x_0)|^{p-2} \partial_\nu G_{\lambda_*}(x,x_0) \right) d\sigma(x)= 0 
\end{eqnarray*}
and then $H_{\lambda_*}(x_0,x_0)=0$ by \eqref{pohozaevGepsilonlimite}.


\bibliographystyle{plain}

\end{document}